\documentclass[12pt]{article}

\usepackage{amstext    }
\usepackage{amsthm    }
\usepackage{a4}
\usepackage[mathscr]{eucal}
\usepackage{mathrsfs}
\usepackage{hyperref}

\usepackage{amsmath}
\usepackage{amssymb}
\usepackage{amscd}

\newtheorem{theorem}{Theorem}[section]

\newtheorem{proposition}[theorem]{Proposition}
\newtheorem{corollary}[theorem]{Corollary}

\newtheorem{remark}[theorem]{Remark}
\newtheorem{remarks}[theorem]{Remarks}
\newtheorem{example}[theorem]{Example}
\newtheorem{examples}[theorem]{Examples}

\newcommand{\supp}{{\rm supp}}

\newcommand{\dist}{{\rm dist}}

\newcommand{\dbar}{{\overline\partial}}
\newcommand{\ddbar}{{\partial\overline\partial}}

\newcommand{\C}{\mathbb{C}}

\newcommand{\R}{\mathbb{R}}

\renewcommand{\P}{\mathbb{P}}

\title{Pfaff systems, currents and hulls}

\author{Nessim Sibony}

\begin{document}

\maketitle

\begin{abstract}
Let  $S $ be a  Pfaff system of dimension 1, on  a compact complex manifold $M$.
  We prove that there is a positive $\ddbar$-closed
 current $T$ of bidimension $(1,1)$ and of mass $1$  directed by the Pfaff system $S.$ There is  no  integrability assumption. 
 We also show that local singular solutions always exist. Under a transversality assumption of $S$
  on the boundary of an open set $U,$ we prove the existence in $U$ of positive $\ddbar$- closed
  currents directed by $S$ in $U.$
 
 Using  $i\ddbar$-negative currents, we discuss Jensen measures,
 local maximum principle and hulls with respect to a  cone  $\mathcal P$ of smooth functions
 in the Euclidean complex space, subharmonic in some directions. The case  where $\mathcal P$
 is the  cone of plurisubharmonic  functions is  classical. We use the results to describe the harmonicity properties of  the solutions of  equations of homogeneous, Monge-Amp\`ere type. We also discuss extension  problems of positive directed currents.
  \end{abstract}

\noindent
{\bf Classification AMS 2010:} Primary: 37A30, 57R30;  Secondary: 58J35, 58J65, 60J65.



 \section{Introduction} \label{intro}

 
 Positive currents play  an important role in complex  dynamics, especially in the  theory of singular foliations
 by Riemann surfaces. The analog of invariant measures for  discrete dynamical  systems turns out to be the positive  $\ddbar$-closed
 currents directed by the foliation. The  existence of such currents is proved  in \cite{BerndtssonSibony}.  See also  \cite {S}, \cite{FornaessSibony}
 and \cite{DinhNguyenSibony}. 
 
 Surprisingly  enough, it turns out that  such currents
  do exist without  any integrability condition.

 \begin{theorem}\label{T:Pfaff_0}
Let $(M,\omega)$ be a compact complex Hermitian manifold  of dimension $k.$ Let $(\alpha_1,\ldots ,\alpha_{k-1})$ be continuous $(1,0)$-forms on $M.$ Then, there exists  a positive current $T$ on $M,$
 of mass $1$ and of bidimension $(1,1),$  such that $T\wedge \alpha_j=0$ for $j=1,\ldots, k-1,$
and $i\ddbar T=0.$  
\end{theorem}

  The proof is  a  variation of the  proof in   \cite{BerndtssonSibony}. This  should  be the first step in order to study the  global
  dynamics of some Pfaff systems.
  
  In the second  part of the  article, we revisit some results  about polynomial  convexity and hulls from \cite{DuvalSibony}.
  We develop the study of convexity in the  context of  $\Gamma$-directed currents. The main new tool  is a refinement of a maximum principle proved in
  \cite{BerndtssonSibony} for plurisubharmonic functions with respect to a  positive current $T,$  such that $i\ddbar T\leq 0.$ 
  
For simplicity  we study the convexity  theory mostly  in $\C^k.$ For $z\in\C^k,$
  let $\Gamma_z$ be a closed cone  of $(1,0)$-vectors and assume  $\Gamma:=\cup_{z\in\C^k} \Gamma_z$ is  closed.
  
  Let $\mathcal P_\Gamma$ denote the cone of smooth functions $u,$ such that  for all $\xi_z\in\Gamma_z,$ $$\langle i\ddbar u(z), i\xi_z\wedge \bar\xi_z  \rangle\geq 0.$$  When for all $z,$ $\Gamma_z$  is the cone  of all  $(1,0)$-vectors, we get  $\mathcal P_0$ the cone
  of smooth plurisubharmonic functions (psh for short).
  
  It is natural to introduce  the polar  of  $\mathcal P_\Gamma,$  i.e. the positive  currents $T,$ of bidimension $(1,1),$  such that for every cutoff function $\chi$, and for every $\ \text{for every}\  u\in\mathcal P_\Gamma ,\ \langle \chi T, i\ddbar u\rangle \geq 0\ .$
  
   Such  currents are the 
  $\Gamma$-directed currents. Notions of positive $\Gamma$-directed currents are  used  in  \cite {S}
  and \cite{BerndtssonSibony}.
  
  There is  a natural  notion  of hull  of a  compact $K$  with respect to $\mathcal P_\Gamma.$
  We define
  $$
  \widehat K_\Gamma:=\left\lbrace z\in\C^k:\ u(z)\leq \sup_K u,\quad\text{for every}\ u\in \mathcal P_\Gamma   \right\rbrace.
  $$
 Let $\mu_z$  be a Jensen measure representing  $z\in  \widehat K_\Gamma,$ with respect to  $ \mathcal P_\Gamma.$ 
   Then, there is a  $\Gamma$-directed current $T_z\geq 0,$ with compact support, such that $i\ddbar T_z=\mu_z-\delta_z.$
  Here $\delta_z$ denotes the Dirac mass at $z.$ This permits to get local singular solutions for Pfaff systems and more generally for  distributions of cones in the tangent bundle. The computation
  $\mathcal P_\Gamma$ hulls appears as a problem in control theory.
  
  The representation of Jensen measures is proved  for $\mathcal P_0$ in  \cite{DuvalSibony}. It has been extended by
  Harvey-Lawson  \cite{HarveyLawson} to hulls in calibrated geometries.
  
  The previous result permits to give a description  of $\widehat K_\Gamma\setminus K$  as a union of supports of currents $T_z.$
  They replace somehow, the analytic varieties  with boundary in  $K,$ which are known  not to  always exist.

We also  prove an  extension  of Rossi's local maximum principle. In the context of polynomial convexity the statement is  as follows.

Let $\widehat K$ denote the polynomial convex hull of $K.$  Let $V$ be a neighborhood of a point 
   $z_0\in \widehat K\setminus K.$ 
Then for every psh function $u,$
$$
u(z_0)\leq  \sup\left\lbrace u(\zeta):\ \zeta\in (\widehat K\cap\partial V)\cup (V\cap K)  \right\rbrace.
$$

Surprisingly  enough, we get a similar  statement for $\mathcal P_\Gamma,$  and a new proof of Rossi's maximum
principle, without  using  any deep  several complex  variables result, as  it is  classical, see \cite{Rossi}  
and \cite{Hoermander}. We use  the maximum principle for appropriate currents. This turns out to be  a   consequence of Green's formula.

To study the above notion of hull, we show that if $\Gamma_z \not=0$ for every $z$ in
an open set $U$, then $U$ is contained in the $\Gamma$-hull of its boundary. As a consequence, under a transversality assumption of  the cones $\Gamma_z$  to $\partial U,$ we prove the existence of  positive $\ddbar$-closed currents directed by $\Gamma$ in $U.$

We then  develop  the Perron  method for  Monge-Amp\`ere  equations in the context  of $\Gamma$-hulls.
  Let $B$ denote the unit ball in $\C^k.$ In the classical context, Bedford and Taylor \cite{BedfordTaylor} proved that any smooth  function $v$ on $\partial B$ can be  extended  as a psh function $u$ in  $B$ and the extension is in   $\mathcal C^{1,1}.$ Moreover,
  $(i\ddbar u)^k=0$ in $B$.  It turns out, that even when $v$ is a real polynomial, the corresponding function $u$
  is not necessarily of class $\mathcal C^{2},$ see \cite{GamelinSibony}. The fact that $u$ is only in $\mathcal C^{1,1},$ is crucial in applications.
  
   It was  observed by the author, long ago, that  for such functions, there is  not always a non-trivial holomorphic disc
 $D_z$ through $z\in B,$  such that 
$u$  restricted to $D_z$ is  harmonic. The example is based on the existence of non-trivial polynomial hulls without analytic structure. 

A recent paper by Ross-Nystr\"om   \cite{RN} describes a large class of geometric examples, with a similar phenomenon.
This raise the following question. Let  $u$ be a  continuous psh function in the  closed unit ball $\bar B$ of $\C^k.$ Assume  $(i\ddbar u)^k=0$ in $B.$ What are the analytic objects on which, the function $u$ is harmonic?
We show here, that  there is always  a Jensen measure $\mu_z$ and a positive current $T_z,$ of bidimension $(1,1),$ such that  $
 i\ddbar T_z=\mu_z-\delta_z 
 $ in $B$
and 
$i\ddbar u\wedge T_z=0.$  So the function $u$ is $T_z$ harmonic.

 In Section 5, we show that a function  which is locally  a  supremum of functions  of $\mathcal P_\Gamma,$ is  globally  a  supremum of such functions. This type of results is  discussed in \cite{GamelinSibony} in the context
of function  algebras.

In the last section, we introduce $\mathcal P_\Gamma$-pluripolar sets and we study extension Theorems for positive  $\mathcal P_\Gamma$-directed currents.

\medskip\noindent
{\bf Acknowledgements.} It is a pleasure to thank T.C. Dinh and M. Paun for their interest and comments.  
\section{Pfaff equations on complex manifolds}

Denote by $D$ the unit disc in $\C.$  A holomorphic map $\varphi_a:\ D\to M,$ $\varphi_a(0)=a,$ is
called a holomorphic disc. It is non degenerate at $a$ if  the derivative at $0$ is non zero.

\begin{theorem}\label{T:Pfaff}
Let $X$ be 
a compact set in a complex Hermitian manifold $(M,\omega)$ of dimension $k.$ Let $E$ be a  locally pluripolar set  in $M$ (not 
necessarily closed). Assume $\overline{X\setminus E}=X.$
Let $(\alpha_1,\ldots ,\alpha_{m})$ be continuous $(1,0)$-forms on $M.$
Assume that for each $a\in X\setminus E,$ there is  a holomorphic disc,  with image in $X,$
 $\varphi_a:\ D\to X,$ $\varphi_a(0)=a,$ non degenerate at $a,$
such that $(\varphi_a^*\alpha_j)(a)=0,$ for $j=1,\ldots, m.$ Then there exists  a positive current $T,$
supported on $X,$
 of mass $1$ and of bidimension $(1,1),$  such that $T\wedge \alpha_j=0$ for $j=1,\ldots, m,$
and $i\ddbar T=0.$  
\end{theorem}

 \begin{proof}
 The difference with Theorem 1.4 in \cite{BerndtssonSibony} is that we do not assume here that the 
 holomorphic disc $\varphi_a$ satisfies $ (\varphi_a^*\alpha_j)(z)=0,$ for every $z$ in $D. $ So we do not assume that the system $S$ has local solutions, which are holomorphic discs.
 
 Recall that $E$ is locally pluripolar, if for every $a\in M,$ there is   a neighborhood $U$ of $a$ and a psh  function 
$v$ in $U$ such that $E\cap U\subset \{z\in U:\  v(z)=-\infty\}.$
Let 
\begin{multline*}
C:=\left\lbrace  T:\  T\geq 0\ \text{bidimension}\ (1,1),\ \text{supported on}\ X,\ \int T\wedge \omega=1,\right.\\
\left.   T\wedge \alpha_j=0,\  1\leq j\leq m   \right\rbrace 
.\end{multline*}
Since  the $\alpha_j$ are continuous, $C$ is a  convex compact set in the space of bidimension $(1,1)$-currents.
Let 
$$ Y:=\left\lbrace i\ddbar u,\  u\ \text{test smooth function on M}   \right\rbrace^\perp.$$  The space
$Y$ is the space of the $i\ddbar$-closed currents.  We have  to show that $C\cap Y$ is  nonempty. Suppose  the contrary. The convex compact set
$C$ is in the dual of a reflexive space. The  Hahn-Banach theorem implies that we can strongly separate $C$ and $Y.$
Hence, there is $\delta >0$  and a  test function  $u$ such that $\langle i\ddbar u  ,T \rangle\geq \delta$ for every $T\in C.$
There is  a point $z_0\in X$  where $u$ attains its maximum. Choose  $r>0,$ such that
$ E\cap B(z_0,r)\subset \{v=-\infty\},$ $v$ being psh in $B(z_0,r).$
Then the function $u_1(z):=u(z)-{(\delta/4)}  |z-z_0|^2$  has a maximum  in $B(z_0,r),$ at $z_0.$
We get  also, that for every Dirac current in $C,$ $\langle i\ddbar u_1 ,T \rangle\geq \delta/4.$ 

For $\epsilon>0$ small enough, $u_2:= u_1+\epsilon v$ will still have   a  maximum at a point 
 $a\in X\setminus E,$ near $z_0.$ Consider a non degenerate holomorphic   disc  with center at $a,$  contained in $X,$  such that 
$\varphi^*_a(\alpha_j(a))=0,$ for $1\leq  j\leq  m.$
So the disc is tangent at $a,$ to the  distribution $(\alpha_j)_{1\leq j\leq m}.$  
Then $u_2\circ\varphi_a$ has  a maximum at $0.$ Since $\varphi_a $ is non degenerate at 
$a,$  $i\ddbar (u_1\circ \varphi_a)>0$ at $0,$ and  hence  in a neighborhood of $0.$  Indeed, we just have to apply the above estimate to a positive Dirac current directed by $ \varphi'_a (0).$ Adding $v,$ preserves the strict subharmonicity. This  contradicts the local maximum principle for the subharmonic function  $u_2\circ\varphi_a$  in a disc.  Hence $C\cap Y\not=\varnothing.$ The theorem is proved.
\end{proof}

\noindent {\it Proof of Theorem \ref{T:Pfaff_0}.}
It is clear that non degenerate, holomorphic discs  in  directions  defined  by  $\alpha_j(a)=0,$ $1\leq  j\leq  k-1,$
always exist.
\hfill $\square$

\begin{remark}\rm
The condition  on the  continuity of  $(\alpha_j)_{1\leq j\leq  k-1}$
can be relaxed. It is   enough  to  assume that $z\mapsto i\alpha_j\wedge \bar\alpha_j$ is  lower-semi continuous (lsc for short)
in the convex salient cone of positive $(1,1)$ forms. Then   if $T_n\in  C$ converge to $T,$
$$
\int T \wedge i\alpha_j\wedge \bar\alpha_j \chi \leq \liminf_n \int T_n \wedge i\alpha_j\wedge \bar\alpha_j \chi=0 
$$
for any positive  cutoff  function $\chi.$ It follows that  the convex set $C$ in the  proof of Theorem  \ref{T:Pfaff},
is  closed.
\end{remark}
\begin{examples}\rm
\noindent 1.  Let $u$ be  a function in $\P^2$ of class  $\mathcal C^2.$ Consider the positive  current $S:= i\partial u\wedge \overline{\partial u}.$ 
It is   directed  by the Pfaff form $\alpha:=\partial u.$ 
It is   shown in \cite{FornaessSibony} that if $u$ is  real  non-constant, then $i\ddbar S$ is not identically zero.
 According to  Theorem \ref{T:Pfaff_0}, there is 
a  current $T$  positive and $\ddbar$-closed, directed by  $ i\partial u\wedge \overline{\partial u}.$ 
This means that  on a set where $ \partial u$  is non-zero, there is a positive  $\sigma$ finite
 measure  $|T| $,
such that on that set
  $$
T=i|T|\partial u\wedge \overline{\partial u}
$$
is  $\ddbar$-closed. In particular if $ i\partial u\wedge \overline{\partial u}$  vanishes only on a finite set of points we get the above representation globally, 

\noindent 2. Let $X$ be  a compact set in $M.$ Assume that $f$ is  a  holomorphic self-map in $M,$  such that $f(X)=X.$
Assume  moreover that  $X$ is  hyperbolic  in the  following dynamical sense. For every $a\in X,$ not contained in a given pluripolar set $E,$ there is  a local holomorphic stable manifold $ W^s_a  $
manifold  of positive dimension, contained in $X$ . The stable manifold $ W^s_a , $ is defined as
$$
W^s_a:=\left\lbrace  z\in M:\  \dist(f^n(z),f^n(a))<C\rho^N,\ N\geq 1  \right\rbrace,
$$
where $\rho <1.$ We assume that $
W^s_a$ contains a  holomorphic disc $ D^s_a$ of positive  radius and that these  discs 
$ D^s_a$ vary continuously.  This  is a  situation  studied in smooth  dynamical systems. The proof of  Theorem  \ref{T:Pfaff} shows that there is $T\geq 0$ directed by these discs  and  $\ddbar$-closed.

An example of this situation is the one of a hyperbolic H\'enon map $f$ in $\C^2$. The set $X $ is the closure in $\P^2$ of  set of points with  bounded orbits. The set $E,$ is just the unique indeterminacy point of the extension of $f$ to $\P^2.$ The situation is studied in detail in  \cite{FS0}.

\noindent 3. Consider  $(M,V,E)$  where $M$ is   a compact complex  manifold, $E$   a closed locally pluripolar set and $V$ is  a  holomorphic sub-bundle of the restriction of the tangent bundle  $T{M}$ to the complement of $E$ in $M$.
 If there is no non-zero positive $\ddbar$-closed current supported on $E$, then, we get positive   bidimension
$(1,1),$   $\ddbar$-closed currents directed by $V.$ They should be useful in the study of the Kobayashi metric.   

\end{examples}

\section{Convexity  and  currents in $\C^k$}

\subsection{$\Gamma$-directed  currents in $\C^k,$  $\Gamma$-hulls}
\label{ss:directed_currents_hulls}
Let  $\mathcal P_0$ denote  the convex   cone  of smooth  psh functions in $\C^k.$
A function $u\in\mathcal P_0$ iff  $\  \langle  i\ddbar u(z), i\xi_z\wedge \bar\xi_z \rangle\geq 0 $  for every $(1,0)$-vector  $\xi_z.$

A function is psh if  $u=\lim\downarrow  u_\epsilon,$  where $u_\epsilon$  are   smooth  psh functions.    We want to study   hulls with respect  to convex  cones  containing $\mathcal P_0.$  We will obtain in this  setting,
new  proofs  of classical results like  Rossi's local maximum principle. All the notions can be developed for an open set  $U\subset \C^k$ or more generally a manifold $N$ admitting a smooth  strictly psh  function $\rho.$

 We now introduce the cone  $\mathcal P_\Gamma.$   For every $z\in \C^k,$ let  $\Gamma_z$ be  a closed cone of $(1,0)$-vectors. We assume that $\Gamma_z$ always contains non-zero vectors and
 that  $\Gamma:=\bigcup_{z\in\C^k}\Gamma_z$  is  closed. Define
$$
\mathcal P_\Gamma:=\left\lbrace u\in\mathcal C^2:\  \langle  i\ddbar u(z), i\xi_z\wedge\bar\xi_z \rangle \geq 0\ \ \text{for every}\  \xi_z\in\Gamma_z \right\rbrace.
$$
The cone   $\mathcal P_\Gamma$ is  convex   and   contains   the cone  of  smooth psh functions.  We define,
$$
\overline{\mathcal P}_\Gamma:=\left\lbrace u:\  u=\lim\downarrow  u_\epsilon,\  u_\epsilon\in  \mathcal P_\Gamma \right\rbrace.
$$ 
We now  consider the   positive currents  of bidimension  $(1,1),$ $\Gamma$-directed. $$
\mathcal P_\Gamma^{\text{0}}:=\left\lbrace T\geq 0, \ \text{of bidimension}\ (1,1),\   T\wedge i\ddbar u\ \geq 0\ \ \text{for every}\  u\in\mathcal P_\Gamma \right\rbrace.
$$
The cone $ \mathcal P_\Gamma^{\text{0}}$ is closed for the weak-topology of currents. It is the closed 
convex hull of the Dirac currents it contains. If $\mathcal P_\Gamma=  \mathcal P_0,$ then $ \mathcal P_\Gamma^{\text{0}}$ are all positive currents. 
In general it is a smaller class of currents.

An important property of the cone 
$\mathcal P_\Gamma$ that  we will use frequently is the following. Let $u_1,...,u_m$ be  functions in  
$
\mathcal P_\Gamma$  and let $H$ be a  smooth function in $ \R^m,$ convex and increasing in each variable. Then $ H(u_1,...,u_m) \in  \mathcal P_\Gamma$. If $H$ is non-smooth, then $ H(u_1,...,u_m) \in \overline{\mathcal P}_\Gamma.$

 Let $T$ be a  positive  $\Gamma$-directed current of bidimension $(1,1).$ A function $v$ is  $T$-subharmonic, if there is a  decreasing
 sequence  $(v_n)$ of smooth functions  such that $i\ddbar v_n\wedge T\geq 0$ and $v_n\to v $ in $L^2(\sigma_T),$ where
 $\sigma_T$ is  the  trace measure  of $T.$ The notion can be localized.
 
 A bounded   function  $u$ is  $T$-harmonic, in an open set $U,$ if there is a  monotone  sequence of smooth functions $u_n,$    
 $|u_n|\leq  C$ in  $U$, such that $\lim\downarrow u_n=u$ or $\lim\uparrow u_n=u,$ and $0\leq  i\ddbar u_n\wedge T\to 0.$ We  will  write
 $   i\ddbar u\wedge T= 0.$
 
 We  will introduce the notion of hull  with respect to  $\mathcal P_\Gamma.$ Let  $K$  be   a  compact set in $\C^k.$ We define  the hull $\widehat K_\Gamma$ of $K$ with respect to $\mathcal P_\Gamma$ as  follows:
 $$
 \widehat K_\Gamma:=\left\lbrace z\in\C^k:\  u(z)\leq \sup_K u\quad \text{for every}\ u\in\mathcal P_\Gamma   \right\rbrace.
 $$
 
 When  $\mathcal P_\Gamma$ is  the cone $\mathcal P_0$ of smooth  psh functions, then $ \widehat K_\Gamma=\widehat K,$ the polynomially
 convex hull of $K.$  We always have that $ \widehat K_\Gamma\subset \widehat K.$ The hull  $ \widehat K_\Gamma$ is  unchanged
 if we replace $\mathcal P_\Gamma$ by $\overline{\mathcal P}_\Gamma.$ 
 
 We  recall  few facts   from Choquet's theory, with some  refinement by Edwards
 \cite{ChoquetMeyer}, \cite{Edwards}. Gamelin \cite{Gamelin} gives an exposition of their results.

On a  compact space $M$,  consider a convex cone $\mathcal E$ of  upper-semicontinuous  (usc for short)  functions  
with values in $[-\infty, \infty). $ We assume that $\mathcal E,$ contains the constants and functions separating points.
 
 A representing measure for $x\in M,$  relatively to $\mathcal E$ is  a  probability measure $\nu_x$ such that
 $\varphi(x)\leq  \nu_x(\varphi)$ for every $\varphi\in\mathcal E.$ It is   clear that  the  set $J_x(\mathcal E)$ of representing  measures
 is  a convex  compact set. The elements of $J_x(\mathcal E)$ are  called  Jensen measures.  Define 
 $$
 \partial_\mathcal E M:=\left\lbrace x\in M:\  J_x(\mathcal E)=\delta_x  \right\rbrace.
 $$ 
 A  result by  Choquet says that every  function in $\mathcal E$ attains its maximum  at a point  of $\partial_\mathcal E M.$ 
 Since  $\mathcal E$ is  convex, there is a  smallest closed set on which every  function  attains its maximum.  It is  called the Shilov 
 boundary $S(M)$ and
 $
 S(M)=\overline{\partial_\mathcal E M}.
 $
 
 It is   easy to check that the set $J_x(\mathcal E)$ is unchanged, if
  we  replace $\mathcal E$, by the  smallest cone containing $\mathcal E$ and  stable by  finite sup. This cone is  also  convex, see  \cite{Edwards}.

\begin{theorem} {\rm (Choquet-Edwards)}  Let $M$ be a compact set and  $\mathcal E$ a  convex cone of continuous  functions from
$M$ to $[-\infty,\infty)$  separating points and  containing the constants.
Let $X$ be  a compact set in $M$   containing the Shilov  boundary  $S(M).$
For  each lsc  function $u:\ X\to (-\infty,\infty],$ define
$$
\hat u(z):=\sup\left\lbrace v(z):\ v\in\mathcal E,\quad v\leq  u\  \text{on}\ X  \right\rbrace.
$$ 
Then
$$
\hat u(z)=\inf \left\lbrace \int ud\mu_z:\ \mu_z\  \textrm{Jensen measure for}\ z\ \text{w.r.t.}\  \mathcal E,\ \textrm{with support in}\ X \right\rbrace.
$$
\end{theorem}

 \subsection{Maximum principle}

 We now  turn to the description of the hull $\widehat K_\Gamma$ in terms of  $\Gamma$-directed  currents.  The  following maximum principle
 which is  a  refinement of  a result in \cite{BerndtssonSibony} will be  useful.
 \begin{theorem}
  Let $U$ be a relatively compact  open set in $\C^k.$ Let $T\in\mathcal P^0_\Gamma.$
 \\
 i) Assume $i\ddbar T\leq 0,$  in $U.$ Then for $z\in \supp (T) \cap U$  we have  for any
 $v\in \overline{\mathcal P}_\Gamma,$
 $$
 v(z)\leq  \sup_{\partial U\cap \supp(T)} v.
 $$
 ii) Assume $i\ddbar T=\nu^+-\nu^-,$  with  $\nu^\pm$ positive measures.  Suppose $u$ is a bounded  $T$-harmonic function.
 Let $v\in \overline{\mathcal P}_\Gamma.$  Assume $v\leq u$ on  $(\partial U\cap \supp(T))\cup (U\cap\supp(\nu^+)). $
 Then $v\leq u$ on $ U\cap \supp(T).$ 
 \end{theorem}
 \begin{proof}
i)  We can  assume that $v$ is  smooth. Suppose $v<0$ on $\partial U\cap \supp(T)$ and that $v(p)>0,$ where $p$ is
  a point in $ U\cap \supp(T).$  The  same inequalities hold if we  replace $v$ by $v_1(z):=v(z)+\epsilon |z-p|^2,$   with $\epsilon >0$
  small  enough. We can assume that $v_1<0,$ in a neighborhood $U_1$ of $\partial U\cap \supp(T).$
  So  the  intersection  of the support of $ v^+_1$  and of $ U\cap \supp(T)$   is compact. Hence
    we get by integration by parts that:
  $$
  0< \int i\ddbar v^+_1\wedge T=\int v^+_1 i\ddbar T\leq 0. 
  $$
 A contradiction. 
 
 We make the argument more explicit. Let $\rho$ be a cutoff function equal to one in a neighborhood of the support of  $ v^+_1$ in $\supp(T)$ and which varies on $\supp(T)$  only on $U_1.$  If we multiply $ v^+_1$ with  $\rho, $ we get a genuine function with compact support.  By integration by part  we obtain the above relation.
  Observe that the function $\rho$ does not appear. Hence the maximum principle holds for ${\mathcal P}_\Gamma$ with respect to $T.$
  
 Since the functions of $ \overline{\mathcal P}_\Gamma,$ are decreasing limits of functions in 
 ${\mathcal P}_\Gamma$  we get also the maximum principle for functions in $ \overline{\mathcal P}_\Gamma.$

 ii) As   above  assume $v-u<0$ on   $(\partial U\cap \supp(T))\cup (U\cap\supp(\nu^+)). $
 Suppose that  for $p\in U\cap \supp(T),$  $(v-u)(p)>0.$ Observe that $  \supp(\nu^+)$ is contained in $\supp(T).$
 
 Let $|u_n|\leq c,$  be a sequence of smooth functions, such that $0\leq  i\ddbar u_n\wedge T\to 0$  and $u=\lim\uparrow  u_n.$ The function $(v-u)$ is usc, hence we can assume that $v$ is smooth.
 So for $n$  large  enough,  we still have 
 $v-u_n<0$ near  $(\partial U\cap \supp(T))\cup (U\cap\supp(\nu^+)).$  Replacing $v,$ by $ v+\epsilon |z-p|^2,$  with $\epsilon>0$ small enough,   we can also   assume that:  $\lim_n i\ddbar (v-u_n)\wedge T\geq \epsilon i\ddbar |z|^2\wedge T.$ Hence
 $$
 0<\lim_n \int i\ddbar(v-u_n)^+\wedge T=\lim_n \int(v-u_n)^+(\nu^+-\nu^-)\leq \lim_n \int (v-u_n)^+\nu^+=0.
 $$
 A contradiction. The case where $u=\lim\downarrow u_n$ is even simpler.  So the maximum principle holds.
  \end{proof}
 
 For the polynomially convex case of the following Corollary see \cite{BerndtssonSibony}and \cite{DL} Proposition 2.5. 
  
  It is proved in \cite{FS1} Corollary  2.6, that the complement of the support
  of a positive current of bidimension (p,p) such that  $i\ddbar T\leq 0$ is locally p- pseudoconvex.
  If $p=1,$ this is just the usual pseudoconvexity. Other notions of pseudoconvexity are used in \cite{DL} to prove similar results.
 
 \begin{corollary}
 Let $K$ be a  compact set in $\C^k.$ Let $T\geq 0$ be a  $\Gamma$-directed current of bidimension $(1,1)$ with compact support..  Assume
 $i\ddbar T\leq 0$ on $\C^k\setminus \ K.$  Then  $\supp(T)\subset  \widehat{K}_\Gamma.$
 \end{corollary}
 \begin{proof}
 Let $B$ be  a large ball containing $\supp(T).$  Let $K_r$ be the$ r- $neighborhood of $K.$
 We can chose for $U$ the set $B\setminus \ K_r$, with $r$ arbitrarily small.
 If $x\in \supp(T)\setminus K$ and   $u\in \overline{\mathcal P}_\Gamma,$  then it follows from the  above maximum principle that $u(x)\leq \sup_K u.$ So $x\in  \widehat{K}_\Gamma.$
  \end{proof}
 
 We now  extend a  result  from   \cite{DuvalSibony} to hulls with respect to $\mathcal P_\Gamma.$
 
\begin{theorem}\label{T:Jensen}
Let $K$ be a compact set in $\C^k,$ contained in a ball $B.$  Assume  $x\in  \widehat{K}_\Gamma.$ Let $\mu_x$  be a  Jensen measure on $K$  representing $x$  for  $\mathcal P_\Gamma.$
 There is a  positive $\Gamma$-directed current $T_x,$ with support in  $ \overline{B},$ such that
 $$
 i\ddbar T_x=\mu_x-\delta_x.
 $$
 In  particular, for every $\varphi\in \mathcal P_\Gamma,$ we have the Jensen's formula:
 $$
 \int \varphi d\mu_x =  \varphi(x)+ \langle T_x,i\ddbar\varphi\rangle.
 $$
 \end{theorem}
\begin{proof}
   Consider the cone $\mathcal P_\Gamma( \overline{B})$ of
smooth functions on $ \overline{B},$ such that, for all $z$ in $ \overline{B},$
 and  $\xi_z\in\Gamma_z$ 
 $$\langle i\ddbar u(z), i\xi_z\wedge \bar\xi_z  \rangle\geq 0.$$ 
 It is easy to check, that $\mu_x$ is a Jensen measure for 
 $\mathcal P_\Gamma( \overline{B}).$ We can assume that $u$ extends smoothly to $\C^k$ and add  a smooth psh function $v$ vanishing on a slightly smaller ball and such that  $u+v$ is in $\mathcal P_\Gamma.$
 Define
$$
C:=\left\lbrace i\ddbar T:\ T \geq 0,\text{} \Gamma\text{-directed}, \ \supp(T)\subset \overline{B} \right\rbrace.
$$
Clearly, $C$ is  a convex set of distributions. Since  $\langle i\ddbar T,|z|^2 \rangle=\langle T, i\ddbar|z|^2 \rangle,$ it follows that any bounded set of
$C$ is  compact. Suppose  $\mu_x-\delta_x\not\in C.$ Then  by Hahn-Banach, there is  a  smooth function $u$ such that
$\int ud\mu_x -u(x)<0.$ Moreover  for every $T\in {\mathcal P}^{\text{0}}_\Gamma,$ supported on $\overline{B},$ 
$\langle i\ddbar T,u\rangle\geq 0.$ So every  $T\in {\mathcal P}^{\text{0}}_\Gamma,$  with $\supp(T)\subset \overline{B}$ satisfies
$\langle  T, i\ddbar u\rangle\geq 0.$

The last inequality says that $u\in\mathcal P_\Gamma( \overline{B}).$ A contradiction, with the fact that $\mu_x$ is a Jensen measure representing $x$ for $\mathcal P_\Gamma( \overline{B}).$

 The Jensen's formula is then clear.   
\end{proof}

\begin{corollary}  We have:  $\widehat {K}_\Gamma=\bigcup \supp(T) )\cup K,$
the first union being taken over all $T\geq 0,$ with compact support, $\Gamma$-directed, such that $i\ddbar T\leq 0,$  in $\C^k\setminus K.$ 
\end{corollary}
\begin{proof}
We have seen that if $i\ddbar T\leq 0,$ on $\C^k \setminus K,$ then $\supp(T)\subset \widehat {K}_\Gamma.$ 
The previous theorem shows that for every $x\in  \widehat {K}_\Gamma\setminus K,$ there is  such a $T$ containing $x$ in its support. 
\end{proof}
\begin{corollary}{\rm (Local maximum principle for   $\widehat {K}_\Gamma$.)}
Let $U\Subset \C^k$ be  an open set. Then  for every $v\in\overline{\mathcal P}_\Gamma$ and $x\in U\cap  \widehat {K}_\Gamma\ $ we have:
$$
v(x)\leq \sup\limits_{(\partial U\cap  \widehat {K}_\Gamma)\cup (U\cap K)} v.
$$
\end{corollary}
\begin{proof}
We can  assume that $x\in U\setminus K.$  Let $\mu_x$ be a Jensen measure for $x,$ supported on $K.$
Let $T_x$ be a  positive $\Gamma$-directed current such that $i\ddbar T_x=\mu_x-\delta_x.$
We then  apply the maximum   principle
 for  $T_x$ and $\mathcal P_\Gamma.$
\end{proof}

Observe that when  $\overline{\mathcal P}_\Gamma=\overline{\mathcal P}_0,$ the cone  of psh functions,
we get Rossi's local maximum  principle  (see \cite{Rossi} or \cite[p.  74]{Hoermander}). Indeed, $\widehat K$ is  the  hull of $K$ with respect  to $\mathcal P_0.$ 
\begin {remark}\rm
For Rossi's local maximum principle, we can also use the  family  ${n\over m} \log |p|,$  with $p$ polynomial and $n,m$ integers, without introducing  general
psh functions and using implicitly that the two notions of hulls are the same, see  \cite{Hoermander}. 
\end{remark}
We show next that under appropriate assumptions a ball is in the $\mathcal P_\Gamma$-hull of its boundary.

\begin{theorem}{\rm (Semi-local solutions for directed systems) }
Let $N$ be  a  complex  manifold, of dimension $k.$ 
Let $\Gamma:=\bigcup_{z\in\ N}\Gamma_z$  be a closed  family of cones as 
in Subsection \ref{ss:directed_currents_hulls} such that  $\Gamma_z \not=0$, for every $z$ in $N.$
 Assume moreover that $N$ admits a smooth exhaustion function 
$\rho$  such that for all non-zero $\xi_z\in\Gamma_z,$ 
$$\langle i\ddbar \rho(z), i\xi_z\wedge \bar\xi_z  \rangle> 0.$$

Let $U$ be  a  relatively compact  open set in $N$ and let $K=\partial U.$  
 Then $ U\subset \widehat K_\Gamma.$ 
In particular  for every $x\in U,$ there is a probability measure $\mu_x$ on $\partial U$ and  a  positive, bidimension $(1,1)$ current $T_x,$ with compact support,
such that 
$$
\qquad i\ddbar T_x=\mu_x-\delta_x.
$$
\end{theorem}
\begin{proof}
  We consider the  cone $\mathcal P_\Gamma$
of smooth $\Gamma$-psh functions. We want  to prove  that, if $K:=\partial U,$ then $\overline U\subset \widehat K_\Gamma.$
Suppose the contrary. Then,  there is $x \in U$ and $u \in \mathcal P_\Gamma,$ such that $u<0$ on  $\partial U$ and $u(x)>0.$
For  $0<\delta\ll 1,$   the function $u_1:=u+\delta \rho$ has the same properties.  Suppose $u_1$ has a  maximum at $a\in U.$
We then   consider a  holomorphic disc $D_a$  tangent at $\xi_a\in \Gamma_a\setminus\{0\}.$ 

The  restriction  of $u_1$
to $D_a$ has a  positive  Laplacean  in a neighborhood of $a,$ and a maximum at $a.$  This is a contradiction.
Hence, $x\in  \widehat K_\Gamma.$  It suffices then to consider a Jensen measure $\mu_x$ supported on $\partial U$ and to adapt the proof of Theorem
 \ref{T:Jensen}.  Using that $\rho$ is an exhaustion function, we can construct first  a $\mathcal P_\Gamma$- convex set containing $\overline U.$  
We can compose $\rho$  with convex increasing functions to get an approximation result as in Theorem  \ref{T:Jensen}.
\end{proof}

\begin{remark}\rm   As a special case we get a similar statement for Pfaff-systems. Let $(\alpha_j)_{1\leq j\leq m}$ be  $m$  continuous $(1,0)$-forms on $N.$  Assume that at any  point $z\in N,$  the  subspace
$$
\Gamma_z:=\left\lbrace  \xi_z:\quad   \alpha_j(z)\xi_z=0,  \quad  1\leq j\leq m \right\rbrace
$$
is  of dimension at least 1. Let $U$ be  a  relatively compact  open set in $N.$
Then for every $x\in U,$ there is a probability measure $\mu_x$ on $\partial U,$ and  a  positive, bidimension $(1,1)$ current $T_x,$
such that 
$$
T_x\wedge \alpha_j=0,\ 1\leq j\leq m\qquad\text{and}\qquad i\ddbar T_x=\mu_x-\delta_x.
$$

\end{remark}

We now address the question of the local existence of $\ddbar$-closed currents. The approach is similar to the one used in  \cite{DuvalSibony} to study the polynomial hull of totally real compact sets.

\begin{theorem}{\rm  }
Let $N$ be  a  complex  manifold, with  a  smooth  positive  exhaustion function $\rho,$ as in Theorem 3.8.
Let  $\Gamma:=\bigcup_z\Gamma_z,$ be  a closed set on the tangent bundle of $N$. We assume that each fiber is a cone and $\Gamma_z \not=0$, for every $z$ in $N.$
Let $U$ be  a  relatively compact  open set in $N,$ with connected and  smooth boundary defined by  $ r<0.$  The function $r$ is smooth in a neighborhood of
$\overline{U}$ and $dr$ does not vanish on $\partial U.$ Assume also that $\overline{U}$ is
$\mathcal P_\Gamma$-convex.

  Assume that there is $z$ in $\partial U$  such that every vector in $\Gamma_z$ is transverse to $\partial U$. Then there is a positive $\Gamma$-directed current $T$ of mass one in $U$ such that $i\ddbar T=0$ in $U.$
  \end{theorem}
\begin{proof}  Let $V$ denote the non-empty open set  of points $a$ in $\partial U$  such that every vector in $\Gamma_a$ is transverse to $\partial U$.  For $a$ in $V$ define $C(a)$ as the convex compact set of  positive $(1,1)$ $\Gamma$- directed currents $S,$ of mass one, supported on $\overline{U}$ and such that 
 $i\ddbar S=\lim (\nu^+_n-\nu^-_n).$ Here, $ \nu^+_n$ are positive measures supported on $\partial U$
 and $\nu^-_n$ are positive measures whose support converges to the point $a.$
 
 For $S$ in $C(a),$ $i\ddbar S$ is supported on $ \partial U.$ We need to show that the currents in
  $C(a)$ are not supported on $\partial U.$ Let $a=\lim a_n$ with $a_n$ in $U$. It follows from 
  Theorem 3.8, that there is a sequence $T'_n$ of positive $(1,1)$  $\Gamma$-directed  currents and a sequence $\mu_n$ of Jensen measures supported on 
  $\partial U$ such that
   $$ i\ddbar T'_n=\mu_n-\delta_{a_ n}.$$ 
   Define $T_n =c_nT'_n,$ where $c_n$ is a constant such that $T_n$ is of mass one. Assume $T= \lim T_n$. We show that $T$ has non-zero mass on $U.$ Suppose on the contrary that $T$ is supported on $\partial U$.
    We prove first that it has zero mass on $V.$ Let $\theta$ be a nonnegative test function supported in a neighborhood of a point $x$ in $V$ and identically one near $x.$

 From the transversality hypothesis on $V$ it follows that:
   $$\langle  \theta T,i\ddbar r^2\rangle= \langle  \theta T,2i\partial r\wedge \overline{\partial r} \rangle, $$
   is strictly positive if $T$ has positive mass near $x$ and $T$ is directed by vectors transverse to $V.$
   The vanishing of $r^2$ to second order on $\partial U$ implies that:
   $$\langle  \theta T,i\ddbar r^2\rangle=  \langle  \theta i\ddbar T, r^2\rangle
    =c_n \langle \mu_n-\delta_{a_ n} , \theta r^2\rangle.
 $$ 
      Since $ \mu_n$ is  supported on the boundary, the limit is non-positive. 
    
    So the  positive current $T$ is supported on the compact  $Z= \partial U\setminus V.$ Since $a$ is not in $Z$, we then have that 
    $i\ddbar T \geq 0.$ This is only possible if $T=0.$ Hence $T$ gives mass to $U.$ Indeed, if $\rho$ is strictly psh negative in a neighborhood of
$Z$, then
$$\langle  T,i\ddbar \rho \rangle =\langle i\ddbar  T, \rho \rangle$$ which is negative. So $T=0.$
 
Similarly, if $a,$ $a'$ are two points in $V,$ then $C(a),$ is disjoint from $C(a').$ If $T$ belongs to the intersection, it is positive and
$i\ddbar T \geq 0.$ Hence it is equal to zero.
    \end{proof}
    
    \begin{remark}\rm 
Let  $K$ be  a compact set on $\partial U.$ Assume  that $ \widehat K_\Gamma$  is non-trivial and contained in $\bar U.$ Assume also that $\Gamma$ is transverse to $\partial U.$ Then there is a positive $\Gamma$-directed current $T$ of mass  $1$, supported on $ \widehat K_\Gamma$   such that $i\ddbar T=0$ in $U.$ Indeed, since currents with negative $i\ddbar,$ cannot be compactly supported in $U,$ there is a sequence  $a_n$ in $U\cap  \widehat K_\Gamma,$ such that $a=\lim a_n$ is in $K.$ We can then apply the scaling in the proof of the Theorem 3.11, in order to construct 
positive currents.

\end{remark}

 \begin{corollary}
 Let $N$ and $U$ be as in Theorem 3.10. Let $\mathcal F$ be a foliation by Riemann surfaces in a neighborhood of $\bar U.$ Assume the leaves are transverse to $\partial U$ on a non empty open set $V.$ Then there exists a positive current $T$ directed by $\mathcal F$ of mass1 on $U,$ satisfying
 $i\ddbar T=0$ in $U.$ 
 
 If the singularities of the foliation are isolated points and each singularity is non degenerate, then the current  $T$ is $L^2$ regular. More precisely, there is a $(0,1)$-form $\tau$ such that $\partial T= \tau \wedge T$ and $\tau$ is $L^2$ integrable with respect to $T$ on 
 $U.$
 \end{corollary}
\begin{proof}
For the first part, it suffices to apply Remark 3.11. For the second part, it suffices to use the decomposition on the flow boxes and to  to observe as in  \cite{FornaessSibony}, thanks to the Ahlfors Lemma, that the metric 
$i\tau\wedge\bar\tau$ is bounded by the Poincar\' e metric.  Moreover  as shown in \cite{DinhNguyenSibony}, Proposition 4.2, the current $T$ has finite mass with respect to the Poincar\'e metric.
The result follows. 
\end{proof}

\begin{remark}\rm
In  \cite{BerndtssonSibony}, $L^2$ estimates for the $\dbar$ equation on $L^2$-regular currents are proved. One can observe that to get the results there, it is enough to assume that the weights $\phi$
are $T$-subharmonic.
\end{remark}

\begin{example} \rm Consider in $\C^4$  the  forms. $$
\alpha:=dz_2-z_3 dz_1,\qquad  \beta:=dz_3-z_4dz_1.$$
The  system  $\{\alpha,\beta\}$  does not satisfy the  Frobenius condition. Moreover through  every point there are infinitely many   holomorphic  discs. Let $f$ be  any holomorphic function. Consider the disc    defined by $z_2:=f(z_1),$  $z_3:=f'(z_1)$
and $z_4:=f''(z_1).$ These  discs are tangent to  $\{\alpha,\beta\}.$ The example is just the complex analog of the  Engel normal form in Engel's condition,
\cite[p. 50]{Bryant}.

More generally it is  enough  to assume that the ideal generated by 2 forms $\{\alpha,\beta\}$  satisfy Engel's condition   at each point, \cite[p. 50]{Bryant}. 
   Let $\Gamma_z$ be the space of vectors $\xi_z$ such that,
$\alpha(z)\xi_z=0,$  $\beta(z)\xi_z=0.$  Our description of hulls applies. The hulls are non-trivial. We then get global dynamical systems of interest.

\end{example}

\begin{example} \rm 
Let $\rho$  be  a  smooth function in $\C^k.$ Let $\Gamma_z:=\ker \partial \rho (z).$ Let  $h$ be a positive smooth function on $\R.$ Define $T = h(\rho) i\partial \rho\wedge \overline\partial \rho.$ If the function $h$ is  non-decreasing and $\rho$ is psh, then $i\ddbar T\leq 0.$ So hulls are non-trivial.

For an arbitrary $\rho$, a function $u$ is in $\mathcal P_\Gamma,$ iff $i\ddbar u\wedge i\partial \rho\wedge \overline\partial \rho$ is positive as a $(2,2)$-form.

If the function $\rho$ is strictly psh, then there are no
non trivial $\Gamma$-holomorphic discs, i.e. non-constant holomorphic discs tangent to the distribution $\Gamma_z.$

\end{example}

\begin{example}\rm
The most obvious $\Gamma$-directed currents  are  $\Gamma$-directed holomorphic  discs, when they exist. More precisely, if $\phi:\  D\to\C^k$
is  holomorphic, with $\phi(\zeta)= z,$  then $\phi'(\zeta)\in\Gamma_z$ for each $\zeta\in D.$ The current $[\phi_*(D)] $ is  $\Gamma$-directed, when it is  of
finite mass. Let  $\lambda$ denote the Lebesgue measure on the unit circle.

One can  also introduce the current  $\tau^\phi:=\phi_*\big(\log {1\over |\zeta|}\big).$ It is   shown in \cite[Example 4.9]{DuvalSibony}  that if $\phi$ is bounded and $\phi (0)=x$, then the current
$\tau^\phi$ has finite mass and:
\begin{equation}\label{e:Jensen}
 i\ddbar \tau^\phi=(\phi)_*(\lambda)-\delta_x.
 \end{equation}
Moreover the mass of $\tau^\phi$ is equivalent to
$$
\|\tau^\phi\|=\int  (1-|\zeta|)|\phi'(\zeta)|^2d\lambda(\zeta).
$$
This is precisely equivalent to the Nevanlinna characteristic of $\phi,$ $T(\phi,1),
$ considered as  a map from  $D$  to $\P^k.$ Jensen's formula gives:
\begin{eqnarray*}
T(\phi,1)= \int_0^1{dt\over t}\int_{D_t}  \phi^*(\omega)&=&\int_{D} \log {1\over |\zeta|} \phi^*(\omega)\\
&=&
{1\over 2\pi} \int_0^{2\pi} \log (1+|\phi|^2)(e^{i\theta})d\theta- \log(1+|\phi(0)|^2).
\end{eqnarray*}

A  family of maps is uniformly  bounded, in the Nevanlinna class, if  there is a constant $C$ such that
for every $\phi$ in the family,
 $\limsup_{r<1} \int_0^{2\pi} \log (1+|\phi|^2)(r e^{i\theta})d\theta \leq C.$ 
 Then the family  $\tau^\phi$  is a relatively compact   family of $\Gamma$-directed currents. This
 permits in particular to prove equation (1).

The currents $\tau^\phi:=\phi_*\big(\log {1\over |\zeta|}\big),$ have been used in Foliation Theory and holomorphic dynamical systems. They permit  to prove ergodic theorems and to study rigidity properties, see  \cite{FornaessSibony},  \cite{DinhSibony}, \cite{DinhNguyenSibony}.
\end{example}
\begin{remark}\rm
 In \cite{Wold},  Wold  gives an independent proof of a weaker estimate of the mass of the currents $\tau^\phi:=\phi_*\big(\log {1\over |\zeta|}\big)$ in the above example. He then deduces 
  Theorem  \ref{T:Jensen} in the context of polynomial hulls as a consequence of  the Bu-Schachermeyer, Poletsky    Theorem  (\cite{Bu,Poletsky}). Their theorem says that if $x$ is in the polynomial hull of a compact $K,$ then for  every  Jensen measure $\nu_x$ on $K,$ there is  a  sequence of $\Gamma$-holomorphic   discs $\phi_n: \overline{D}\to\C^k$
such that $\phi_n(0)=x$ and $(\phi_n)_*(\lambda)\to\nu_x.$ 
  
 Following this approach, Drinovec-Forstneric in \cite{DF}, proved Theorem \ref{T:Jensen} when the cone $\Gamma$ is the null cone in  $\C^3,$ i.e.  a vector $a=(a_1,a_2,a_3)$ is in $\Gamma_z$  iff
$
 \sum_{j=1}^3 a_i^2=0$. Their proof requires to construct first $\Gamma$-discs, satisfying  a Bu-Schachermeyer, Poletsky    type Theorem, for the null-cone. 
\end{remark}


\section{Homogeneous $\Gamma$-Monge-Amp\`ere equation and  Dirichlet problem}


Let $\Omega\Subset \C^k$  be  a  smooth  strictly pseudoconvex  domain. Let $u$ be a continuous function on $\mathcal\partial\Omega.$ Using  Perron's method, Bremermann considered the following
envelope. He defined:
$$
\hat u:=\sup\left\lbrace v:\ v\quad \text{psh in}\ \Omega,\quad v\leq u\quad \text{on}\ \partial\Omega   \right\rbrace.
$$  
It is    easy to show that $\hat u$ is continuous in $\overline\Omega,$ psh
in $\Omega.$ Bedford and  Taylor \cite{BedfordTaylor}   showed that $ (i\ddbar \hat{u})^k=0 $ in $\Omega.$
To  explain the  property, it  was tempting  to think that through every point $x\in\Omega,$ there is a non-trivial holomorphic disc  $D_x,$  
such that $\hat u|_{D_x}$ is  harmonic.  That  would explain  that in some sense   there is    a  ``zero  eigenvalue" for $i\ddbar u.$
The  author has given an  example showing that it is not always the  case. More precisely, let $B $ denote the unit ball in 
$\C^k$. There is   $u\in\mathcal C^{1,1}(\overline B),$ such that
 $ (i\ddbar \hat{u})^k=0 $  and the  restriction of $u$ to any holomorphic disc through zero, is  not harmonic. We show  below, that indeed
the result  becomes true if we  replace discs through $x,$ by positive currents $T,$ such that $ i\ddbar T\leq 0,$ strictly negative at $x.$

 We will  consider a  compact set $K,$  such that  every point of $K$ is in the Jensen boundary for  $\mathcal P_\Gamma.$  More   precisely,
if $\mu_x$ is a  Jensen  measure  for $x,$ supported on $K,$  then $\mu_x=\delta_x.$  Let $u$ be  a  lsc  bounded
function  on $K.$   We define  on $\widehat{K}_\Gamma$  the  function 
$$
\hat {u}_\Gamma(x):=\sup\left\lbrace v(x):\ v\in\mathcal P_\Gamma,\quad v\leq u\quad \text{on}\ K   \right\rbrace.
$$  
We will call $\hat {u}_\Gamma$, the solution of the $\Gamma$-Monge-Amp\`ere problem.
Our assumption  implies that $\hat {u}_\Gamma$ is  lsc.  The Choquet-Edwards theorem implies that
$$
\hat {u}_\Gamma(x)=\inf_{\mu_x\in J_x}\int    ud\mu_x.
$$
So in particular, $
\hat {u}_\Gamma=u$   on $K.$

\begin{theorem}  Let $K,$ $\widehat{K}_\Gamma,$  $u,$ $\hat {u}_\Gamma$ as   above. Then for every $x\in\widehat{K}_\Gamma\setminus K,$
there is  a Jensen measure $\mu_x,$ and a positive current  $T_x\geq 0,  \Gamma$-directed, of bidimension $(1,1),$ such that  $
 i\ddbar T_x=\mu_x-\delta_x$  in $\C^k\setminus K.$ Moreover, $i\ddbar\hat {u}_\Gamma\wedge T_x =0$ on 
$\widehat{K}_\Gamma\setminus K.$ 

 Conversely, let $\tilde u$ be  a  lsc function on $\widehat{K}_\Gamma$  such that
${\tilde{u}}=\lim\uparrow  u_n,$ where each $u_n$ is a finite sup of functions in $\mathcal  P_\Gamma.$ Assume  that for every $x\in \widehat{K}_\Gamma\setminus K,$  there is 
a positive  current $S_x$ with $x\in \supp(S_x),$  such that $i\ddbar S_x\leq 0$ on  $\C^n\setminus K,$  and   such that $i\ddbar \tilde u\wedge S_x=0$ on $\widehat K\setminus K.$  Then $\tilde u=\hat {u}_\Gamma,$ with $u:=\tilde u|_K.$
\end{theorem}
\begin{proof}
The  Choquet-Edwards theorem implies that for every $x\in  \widehat{K}_\Gamma\setminus K,$ there is  a Jensen measure $\mu_x$ on $K,$
such that $ \hat {u}_\Gamma(x)=\int ud\mu_x.$ Theorem \ref{T:Jensen} implies that there is a
  $\Gamma$-directed current $T_x\geq 0,$ such that $i\ddbar T_x=\mu_x-\delta_x.$   We show that $\langle i\ddbar T_x,  \hat u_\Gamma\rangle=0.$ The function  $\hat u_\Gamma$ is an increasing limit of  a  sequence $v_n\in\mathcal P_\Gamma.$ 
 Hence $\lim_n \langle i\ddbar T_x,   v_n\rangle=0,$  and from Jensen's formula, $0\leq   T_x\wedge i\ddbar  v_n\to 0.$ 
 It follows that, $i\ddbar \hat u_\Gamma\wedge T_x =0.$ We have used implicitly that we can replace the supremum
 by a composition with an appropriate smooth convex function.

For the  converse, let $u:=\tilde u|_K.$  Let $v\in\mathcal P_\Gamma,$  such that $v\leq \tilde u$ on $K.$ Since $\tilde u$  satisfies
the  hypothesis of the maximum  principle, with respect  to $S_x,$  we  get that $v\leq \tilde u$ on $S_x.$  So $ \hat {u}_\Gamma\leq \tilde u.$
It is   clear that $\tilde u\leq \hat {u}_\Gamma .$ 
\end{proof}

\begin{remarks} \rm 1.  In the classical Monge-Amp\`ere case, i.e when  $\mathcal P_\Gamma=\mathcal P_0,$ it is   enough to assume that $\bigcup \supp(S_x)$ is  of full measure
in the domain $\Omega$.  We then  get  $\hat u= \tilde u$ on a set of full measure and hence   
$\hat u= \tilde u$ by  pluri-subharmonicity.

2.   Assume   that  for  every  Jensen measure $\nu_x$ on $K,$ there is  a  sequence of $\Gamma$-holomorphic   discs $\phi_n: \overline{D}\to\C^k$
such that $\phi_n(0)=x$ and $(\phi_n)_*(\lambda)\to\nu_x.$ Here  $\lambda$ denotes the Lebesgue measure, of mass one, on the unit circle.  According to  \cite{Bu,Poletsky}, this  is  the case when $\mathcal P_\Gamma=\mathcal P_0.$ We then have  the following interpretation of the harmonicity of $\hat u$ on $T_x,$ in the context of the classical Monge-Amp\`ere equation.

  Let $\mu_x$  be a Jensen measure on $\partial\Omega,$ such that 
$\int\hat ud\mu_x -\hat u(x)=0.$ Let $\phi_n$ the sequence of holomorphic maps associated to $\mu_x$.  In the present situation, it is easy to show that the function $\hat u,$ is continuous. Then, using  Jensen's formula we get:

\begin{eqnarray*}
\limsup_n\int_0^1 {dt\over t}\int_{D_t} i\ddbar (\hat u\circ \phi_n)&=&\lim_n \int_0^{2\pi} (\hat u\circ \phi_n)(e^{i\theta})d\theta-\hat u(x)\\
&=&\int\hat ud\mu_x -\hat u(x)=0.
\end{eqnarray*}
Observe that $i\ddbar (\hat u\circ \phi_n)\geq 0.$ Hence we  see that $\hat u$ is asymptotically harmonic on the holomorphic
discs $\phi_n:  \overline{D} \to\C^k.$

\end{remarks}


\section{Localization}


We  now  address a  localization problem. Let $v $ be a continuous  function on $\widehat{K}_\Gamma.$  Assume that $v$ is locally    a  supremum  of  functions
in $\mathcal P_\Gamma.$ Is it  globally   the  supremum   of  functions
in $\mathcal P_\Gamma$  ? In particular, is the $\Gamma$-Monge-Amp\`ere problem, a local problem?

More generally, consider a convex cone $\mathcal P$  of $\mathcal C^2$-functions in $\C^k,$ containing  a smooth strictly psh function, the constants and  separating points. We assume also that if $u\in \mathcal P$ and $h$ is  a convex increasing function, then
$h\circ u\in\mathcal P.$ This is a crucial assumption on $\mathcal P.$ For a compact $K,$  we can define the hull  $\widehat{K}_P$ of $K$ with respect to  $\mathcal P.$ We will assume that $\widehat{K}_P,$ is always compact. A basic example is as follows. 

Let $R$ be a strongly positive $(p, p)$ form or current. Define  the cone of functions,
$$
\mathcal P_R:=\left\lbrace u\in\mathcal C^2:\   i\ddbar u\wedge R \geq 0 \right\rbrace
.$$
The cone $\mathcal P_R$ satisfies the following property. Let $u_1,...,u_m$ be  functions in  
$
\mathcal P_R$  and let $H$ be a  function in $ \R^m,$ convex and increasing in each variable, then $ H(u_1,...,u_m)$ is in $\mathcal P_R.$ We can  consider 

$$
\mathcal P^0:=\left\lbrace  T:\  T\geq 0,\quad  \langle \chi T, \ddbar u\rangle \geq 0  \quad\text{for every}\ u\in\mathcal P   \right\rbrace, $$
where $\chi$ is an arbitrary cutoff function. The elements of $\mathcal P^0$ are bi-dimension $(1,1)$
currents. We describe a very special case of the above situation.

 Suppose,  the current $R$ is strongly positive and closed, with no non-constant holomorphic discs in its support \cite{DuvalSibony}. Then there will be no non-trivial holomorphic disc $\varphi:\ D\to \C^k,$ 
such that $u\circ \varphi$ is subharmonic for every $u$  in $\mathcal P_R.$ But hulls are non-trivial. 
The case where $R$ is strongly positive and $i\ddbar$- negative, in an open set  $U,$ is also of interest.

The  previous  theory is  valid for $\mathcal P.$
The maximum principle for $\mathcal P,$ with respect to $T\in\mathcal P^0$, is   valid if $i\ddbar T\leq 0.$ In particular, if we consider the cone $\mathcal P_R$ and we assume that $i\ddbar R\leq 0,$ the 
maximum principle holds for functions on support $R$ which are decreasing limits of functions in 
$\mathcal P_R.$ 

 If $x\in  \widehat{K}_P,$  we have Jensen measures $\mu_x$ for $x,$  with support in $K.$  We can also  solve the  equation $i\ddbar T=\mu_x  -\delta_x,$ with $T\geq 0,$ in  $\mathcal P^0.$ Hence  get the local  maximum  principle for    $\widehat{K}_P. $

 We now introduce the Monge-Amp\`ere  problem, with respect  to $\mathcal P.$  Let $u$ be 
a  bounded  lsc  function  on $K.$ Define $\hat u$ on   $\widehat{K}_P$ as
$$
\hat u:=\sup\left\lbrace v:\ v\leq u\quad\text{on}\quad K,\quad v\in \mathcal P   \right\rbrace.
$$
\begin{proposition}\label{P:Dirichlet}
Let $K$ be  a  compact set in $\C^k$ and let $u$ be a bounded  lsc   function  on $K.$
Let $U$ be  an  open set in $\widehat{K}_P$. Define  $v:=\hat u|_{\partial U\cup (U\cap K)}.$ Let $\hat v$ denote the solution of the
Monge-Amp\`ere problem on $U,$  with respect to $\mathcal P|_U$. Then $\hat v=\hat u$ on $U.$
 In particular, if  for $\varphi\in\mathcal P,$ $\varphi\leq \hat u$
on $\partial U\cup (U\cap K),$ then $\varphi\leq \hat u$ on $U.$
\end{proposition}
\begin{proof}
Let $\varphi\in\mathcal P.$ Suppose  $\varphi\leq \hat u$
on $\partial U\cup (U\cap K).$   We have to  show that  $\varphi\leq \hat u$ on $U.$ Let $x\in U\setminus K $ and  $\mu_x$ a Jensen
measure for $x$ with support  on $K.$ Let $T_x\in\mathcal P^0,$  $T_x\geq 0,$ $i\ddbar T_x=\mu_x-\delta_x,$  be as in Theorem 4.1.
 Since $\hat u$ is $T_x$-harmonic and   $\varphi\leq \hat u$ on $\partial U\cup (U\cap K),$  the maximum principle implies that $\varphi(x)\leq \hat u(x).$
\end{proof}

The  following result  was  proved in the  context of function algebras in  \cite{GamelinSibony}.
\begin{theorem}
Let $u$  be a  continuous function on  $\widehat{K}_P.$ Assume that  every point of $K$ is  a Jensen boundary point for $\mathcal P.$
Suppose that for every  $x\in \widehat{K}_P,$ there is  a neighborhood $U$  such that  on $U,$  $u$ is   a  supremum  of functions in 
$\mathcal P|_U.$ Then $u$ is   a  supremum on  $\widehat{K}_P$  of functions in $\mathcal P.$
\end{theorem}
\begin{proof}
Define for $z\in \widehat{K}_P,$
$$
\hat u(z):=\sup\left\lbrace v(z):\  v\in\mathcal P,\quad v\leq u\quad\text{on}\quad   \widehat{K}_P \right\rbrace  .
$$
We want to show that $\hat u=u.$ The  Choquet-Edwards Theorem implies that $\hat u=u$ on $K.$

Let $\alpha:=\sup\{ u(z)-\hat u(z):\ z\in   \widehat{K}_P\}$  and $E:=\{z:\  z\in   \widehat{K}_P:\ u(z)-\hat u(z)=\alpha  \}.$
Suppose $\alpha >0.$ We need to find  a  contradiction. The  set $E$ is compact. Let $z_0\in E$ be  a  point in the Jensen  boundary of
 $\mathcal P|_E.$ By hypothesis, there is  a  neighborhood $U$ of $z_0$ such that for $z\in U,$
 $u(z)=\sup\{v(z):\quad v\in\mathcal P|_U \}.$ Since $z_0$ is  in the Jensen  boundary of $\mathcal P|_E,$ there is $v_0\in\mathcal P,$
with $v_0<0$ on $E\setminus U,$ and $v_0(z_0)>0.$

 We can choose $c>0$ large  enough such that
$-c\alpha+\sup v_0<0$ on  $\widehat{K}_P$ and $c(u-\hat u-\alpha)+v_0<0$ on $\partial U.$ So  $c(u-\alpha)+v_0<cu$ on  $\widehat{K}_P$  and $c(u-\alpha)+v_0<c\hat u$ on $\partial U.$

The function, $c(u-\alpha)+v_0$ is  a supremum on $U$ of functions on $\mathcal P|_U$ and it is  dominated by $c\hat u$  on $\partial U.$
We can apply Proposition   \ref{P:Dirichlet} for $K= \widehat{K}_P.$ We get that
$c(u-\alpha)+v_0<c\hat u$ on $ U.$ This contradicts that $v(z_0)>0$ and $c(u-\alpha-\hat u)(z_0)=0.$
\end{proof}


\section{Extension of directed currents}


The problem of extension of positive closed currents through analytic varieties or through complete pluripolar set has been studied quite extensively,
see \cite {Skoda},\cite {ElMir}, \cite{Sibony2}, \cite{DS1}. 

 Here we replace , the classical complete pluripolar  sets, by complete pluripolar sets with respect to
$\mathcal P.$ A positive current 
 $T$ of bi-degree $(p,p)$ is $\mathcal P$-directed iff  the current $T \wedge i\ddbar u, $ is positive for every $u$ in $\mathcal P.$ 
 We obtain extension results for positive $\Gamma$-directed currents, $\ddbar$-closed and of bounded mass in the complement of a closed complete  $\mathcal P_\Gamma$-pluripolar set.

In this section we  first discuss the notion of  pluripolar sets with respect to a cone $ \mathcal P.$ 
We then we give few  extension results, see Theorem 6.4.

On an open connected set $U$ in $\C^k,$ we consider a cone $\mathcal P$ of smooth functions, as in section 5. The two main cases, we have in mind are the cones  $\mathcal P_\Gamma$ considered in section 3, and the cones
 $\mathcal P_R,$ associated to a positive current $R$, as in section 5. We define
$$
\overline{\mathcal P}:=\left\lbrace u:\  u=\lim\downarrow  u_\epsilon,\  u_\epsilon\in  \mathcal P \right\rbrace .
$$ 
We will consider only functions  which are not identically $-\infty\ .$ 

 A set $F$ in $U$ is  $\mathcal P$-pluripolar, if it is contained in  $ \{z\in U:\  v(z)=-\infty\}, $  where $v$  is a function   in $\overline{\mathcal P}. $ The set $F$ is complete $\mathcal P$-pluripolar, if $F = \{z\in U:\  v(z)=-\infty\}.$ We can also define the notions of locally $\mathcal P$-pluripolar, or locally  complete $\mathcal P$-pluripolar.
\begin{example} \rm 
Let $\rho$  be  a  smooth non-negative function in $U.$ Let $\Gamma_z:=\ker \partial \rho (z).$  Assume that the function
 $\rho$  is in $\mathcal P$ , i.e.  $  i\ddbar  \rho\wedge  i\partial \rho\wedge \overline\partial \rho $ is positive.
 Then the set $ \rho=0,$ is  $\mathcal P$-pluripolar. It is associated to the function $\log(\rho),$ which is in 
$\overline{\mathcal P}. $ Similarly for any  $c>0$ the set $( \rho\leq c),$ is  $\mathcal P$-pluripolar. It is associated to the function $(\rho-c)^+.$
  So the $\mathcal P$-pluripolar sets, can be quite different from the classical ones.
\end{example}

As in \cite{Sibony2}, we have the following result.

\begin{proposition} \label{prop_pluripolar}
Let $F$ be a complete $\mathcal P$-pluripolar set in $U.$ For any $V\Subset U,$ there is a sequence $(u_n)$ of 
functions in  $\mathcal P|_V,$  $0\leq u_n\leq 1,$ vanishing in a neighborhood of $ V\cap F,$  and such that, $(u_n)$
converges pointwise to 1 in $U\setminus F$.  When $F$ is closed we can choose $(u_n)$ increasing.
\end{proposition}
\begin{proof}
Choose a sequence $(v_n)$ in  $\mathcal P,$ decreasing to a function $v$  in $\overline{\mathcal P},$ such that 
$F = \{z\in U:\  v(z)=-\infty\}.$
We can assume that  $v_n \leq 0$ on  $V.$ Let $h$ be a convex increasing function on $\R$ vanishing near $0$
and such that $h(1)=1.$ 

 Define $ u_j= h (exp(v_j/n_j)), $ with $n_j$ increasing to infinity. The sequence $ (u_j)$
satisfies the stated properties. If $F $ is closed, we can first construct a function $u$ in $\overline{\mathcal P},$ smooth out of $F$, continuous and such that  $F = \{z\in U:\  u(z)=-\infty\}.$ Then $ u_n= h (exp(u/n)), $ satisfies the required properties.
\end{proof}

We start with the following proposition which is a version of the
Chern-Levine-Nirenberg inequality \cite{ChernLevineNirenberg},
 \cite {Demailly2}.
In what follows, $\omega$ denotes a K\"ahler $(1,1)$-form on 
 $U.$  If $T$ is a current of order zero, the mass of $T$ on a Borel set
$K\subset U$ is denoted by $\|T\|_K$. If $T$ is a positive or a negative 
$(p,p)$-current, $\|T\|_K$ is equivalent to
$|\int_K T\wedge \omega^{k-p}|$. We identify these two quantities.
Observe however that the mass estimates (resp. the extension results) for 
$(p,p)$-currents, can be easily reduced to similar question for bi-dimension $(1,1)$-currents. It is enough to wedge, with an appropriate power of
$\omega.$  This is also valid for $\mathcal P$-directed currents. 

\begin{proposition} \label{prop_cln}
Let $U$ be an open subset $\C^k.$
Let $K$ and $L$ be compact sets in $U$ with $L\Subset K$. 
Assume that $T$ is a positive current  on $U$ of bi-degree $(p,p),$ $\mathcal P$-directed. Assume also, that $i\ddbar T$ has order zero. 
Then there exists  a constant $c_{K,L}>0$ such that
for every smooth function $u$ in $\mathcal P,$ we have the following estimates
$$\int_L  i\partial u\wedge \dbar u \wedge T \wedge \omega^{k-p-1}
\leq c_{K,L} \|u\|_{ \mathcal L^\infty(K)} ^2\big( \|T\|_K +  \|\  i\ddbar T\|_K \big),$$
and
$$\| i\ddbar u \wedge T\|_L\leq c_{K,L}\|u\|_{ \mathcal L^\infty(K)} \big( \|T\|_K + \|\ i\ddbar T\|_K  \big),$$
where $c_{K,L}>0$ is a constant independent of $u$ and $T$. 
\end{proposition}
The Proposition is proved in \cite {DS1} for psh functions, see also \cite{Sibony2}. The proof can be easily adapted 
to functions in $\mathcal P$ and to  $\mathcal P$-directed currents. 

\begin{theorem} \label{th_extension}
Let $F$ be a
closed subset in $U$. Let $T$ be a positive
$(p,p)$-current $\mathcal P$- directed on $U\setminus F$.
Assume that $F$ is  locally complete  $\mathcal P$-pluripolar and that
$T$ has locally finite mass near $F$. Assume also that there exists a positive 
$(p+1,p+1)$-current $S$ with locally finite mass near $F$ such that $ i\ddbar T\leq S$ on $U\setminus
F$. Then $i\ddbar T$ has locally finite mass near $F$.  If 
$\widetilde T$ and $\widetilde{\ddbar T}$ denote
the extensions by zero of $T$ and $\ddbar T$ on $U$, then  the current 
$i\widetilde{\ddbar T} - i\ddbar \widetilde T$ is positive. If moreover the current $dT,$ is of order zero
and of bounded mass near $F,$ then  $\widetilde{d T}= d  \widetilde{T}.$
\end{theorem}
\begin{proof}

The Theorem is proved in  \cite {DS1} when  $\mathcal P,$  is the cone of psh functions.
The needed modifications are straightforward. We just sketch the strategy. We consider the sequence 
$(u_n)$ of functions constructed in Proposition 6.2. We have $u_n T\rightarrow \widetilde T.$ The following formula is an elementary calculus identity:

$$  u_n(i\ddbar T) - i\ddbar u_n\wedge T=i\ddbar(u_nT)-i\partial( \overline\partial u_n\wedge T) + i \overline\partial ( \partial u_n\wedge T).$$
The previous identity implies that 
$$ u_n(i\ddbar T-S) - i\ddbar u_n\wedge T= -u_n S+ i\ddbar(u_nT)-i\partial( \overline\partial u_n\wedge T) + i \overline\partial( \partial u_n\wedge T).$$
One  shows, using Proposition 6.3,  that $\partial u_n\wedge T\rightarrow 0$.
Then, the right hand side converges to $-\widetilde S+i\ddbar \widetilde T$ where $\widetilde S$ is the trivial extension
by zero of $S$ on $U$. Since both terms on the 
left hand side are negative currents, their limit values are
negative. We then deduce that
$-\widetilde S+i\ddbar \widetilde T$ is negative, and hence
$i\ddbar T$ has bounded mass near $F.$
It follows that $ u_n(i\ddbar T) - i\ddbar(u_nT) =  i\ddbar u_n\wedge T +o(1). $

So, the left hand side of the previous equation, converges to
$ i\widetilde{\ddbar T}-\ddbar \widetilde T$. 
Since  $\lim i\ddbar u_n\wedge T$  is positive, it follows  that :
 $ \widetilde{ i\ddbar T} - i \ddbar \widetilde T$ is positive. The others assertions are proved similarly.

\end{proof}

\begin{corollary} \label{cor_main}
Let  $F$ be a  closed complete $\mathcal P$-pluripolar
subset of $U.$ Let $(u_n)$, be a sequence of functions in $\mathcal P,$ as constructed in Proposition 6.2. Let $T$ be a bi-degree  (p,p) positive $ \ddbar$-closed, $\mathcal P$-directed current  in $U.$ Assume $T$ has no mass on $F.$
Then, for every compact $K$ in $U$ we have that

$$\int_K i\partial u_n \wedge \overline\partial u_n \wedge T \wedge \omega^{k-p-1}
$$
and
$$\int_K i\ddbar u_n  \wedge T\wedge \omega^{k-p-1}
,$$
converge to zero, when $n$ goes to infinity.
\end{corollary}
\begin{proof}

The current  $T$ has no mass on $F$, hence it is equal to its trivial extension through $F.$ It follows that,  $\lim i\ddbar u_n\wedge T$ = $ \widetilde{ i\ddbar T} - i \ddbar \widetilde T=0 .$
So, we get the second relation in the corollary.  If we apply this relation to  the sequence $(u_n^2),$ we get the first assertion.
\end{proof}


\noindent
Nessim Sibony, Universit{\'e} Paris-Sud,\\
\noindent
and Korea  Institute For Advanced Studies, Seoul \\
\noindent
{\tt Nessim.Sibony@math.u-psud.fr},

\end{document}